\pdfoutput=1
\documentclass[11pt,letterpaper]{amsart} 
\usepackage[rgb]{xcolor}
\usepackage{tikz}
\usepackage{verbatim}
\usepackage{amsmath}
\usepackage{amsxtra}
\usepackage{amscd}
\usepackage{amsthm}
\usepackage{amsfonts}
\usepackage{amssymb}
\usepackage{eucal}
\usepackage{float}
\usepackage{mathrsfs}
\usepackage{graphicx}
\usepackage{stmaryrd}
\usepackage{tikz-cd}
\usetikzlibrary{arrows}
\usetikzlibrary{cd}
\usepackage{ytableau}
\usepackage{tikz-cd}
\usetikzlibrary{arrows}
\usetikzlibrary{cd}
\tikzcdset{every label/.append style = {font = \small}}
\allowdisplaybreaks

\usepackage{latexsym,todonotes}

\PassOptionsToPackage{hyphens}{url}\usepackage[linktocpage=true]{hyperref}
\hypersetup{colorlinks,linkcolor=blue,urlcolor=orange,citecolor=blue}
\usepackage[all, cmtip]{xy}

\usepackage{stmaryrd}
\usepackage{color} 

\usepackage{amsrefs}

\newtheorem{thm}{Theorem} [section]

\newtheorem{lem}[thm]{Lemma}
\newtheorem{prop}[thm]{Proposition}

\theoremstyle{definition}
\newtheorem{definition}[thm]{Definition}

\theoremstyle{remark}
\newtheorem{rem}[thm]{Remark}

\numberwithin{equation}{section}

\begin{document}

\newcommand{\thmref}[1]{Theorem~\ref{#1}}
\newcommand{\secref}[1]{Section~\ref{#1}}
\newcommand{\lemref}[1]{Lemma~\ref{#1}}
\newcommand{\propref}[1]{Proposition~\ref{#1}}
\newcommand{\corref}[1]{Corollary~\ref{#1}}
\newcommand{\remref}[1]{Remark~\ref{#1}}
\newcommand{\eqnref}[1]{(\ref{#1})}

\newcommand{\exref}[1]{Example~\ref{#1}}

 \newcommand{\GSp}{\mathrm{GSp}}
 \newcommand{\PGSp}{\mathrm{PGSp}}
\newcommand{\PGSO}{\mathrm{PGSO}}
\newcommand{\PGO}{\mathrm{PGO}}
\newcommand{\SO}{\mathrm{SO}}
\newcommand{\GO}{\mathrm{GO}}
\newcommand{\GSO}{\mathrm{GSO}}
\newcommand{\Spin}{\mathrm{Spin}}
\newcommand{\Sp}{\mathrm{Sp}}
\newcommand{\PGL}{\mathrm{PGL}}
\newcommand{\GL}{\mathrm{GL}}
\newcommand{\SL}{\mathrm{SL}}
\newcommand{\U}{\mathrm{U}}
\newcommand{\ind}{\mathrm{ind}}
\newcommand{\Ind}{\mathrm{Ind}}
\newcommand{\im}{\mathrm{im}}
\renewcommand{\ker}{\mathrm{ker}}
 \newcommand{\triv}{\mathrm{triv}}
  \newcommand{\Tr}{\mathrm{Tr}}
  \renewcommand{\S}{\mathscr{S}}
  \newcommand{\Y}{\mathbb{Y}}
\newcommand{\End}{\mathrm{End}}
\newcommand{\bigzero}{\mbox{\normalfont\Large\bfseries 0}}

\newtheorem{innercustomthm}{{\bf Theorem}}
\newenvironment{customthm}[1]
  {\renewcommand\theinnercustomthm{#1}\innercustomthm}
  {\endinnercustomthm}
  
  \newtheorem{innercustomcor}{{\bf Corollary}}
\newenvironment{customcor}[1]
  {\renewcommand\theinnercustomcor{#1}\innercustomcor}
  {\endinnercustomthm}
  
  \newtheorem{innercustomprop}{{\bf Proposition}}
\newenvironment{customprop}[1]
  {\renewcommand\theinnercustomprop{#1}\innercustomprop}
  {\endinnercustomthm}

\newcommand{\bbinom}[2]{\begin{bmatrix}#1 \\ #2\end{bmatrix}}
\newcommand{\cbinom}[2]{\set{\^!\^!\^!\begin{array}{c} #1 \\ #2\end{array}\^!\^!\^!}}
\newcommand{\abinom}[2]{\ang{\^!\^!\^!\begin{array}{c} #1 \\ #2\end{array}\^!\^!\^!}}
\newcommand{\qfact}[1]{[#1]^^!}

\newcommand{\nc}{\newcommand}

\nc{\Ord}{\text{Ord}_v}

 \nc{\A}{\mathbb A} 
  \nc{\G}{\mathbb G} 
\nc{\Ainv}{\A^{\rm inv}}
\nc{\aA}{{}_\A}
\nc{\aAp}{{}_\A'}
\nc{\aff}{{}_\A\f}
\nc{\aL}{{}_\A L}
\nc{\aM}{{}_\A M}
\nc{\Bin}{B_i^{(n)}}
\nc{\dL}{{}^\omega L}
\nc{\Z}{{\mathbb Z}}
 \nc{\C}{{\mathbb C}}
 \nc{\N}{{\mathbb N}}
 \nc{\R}{{\mathbb R}}
 \nc{\fZ}{{\mf Z}}
 \nc{\F}{{\mf F}}
 \nc{\Q}{\mathbb{Q}}
 \nc{\la}{\lambda}
 \nc{\ep}{\epsilon}
 \nc{\h}{\mathfrak h}
 \nc{\He}{\bold{H}}
 \nc{\htt}{\text{tr }}
 \nc{\n}{\mf n}
 \nc{\g}{{\mathfrak g}}
 \nc{\DG}{\widetilde{\mathfrak g}}
 \nc{\SG}{\breve{\mathfrak g}}
 \nc{\is}{{\mathbf i}}
 \nc{\V}{\mf V}
 \nc{\bi}{\bibitem}
 \nc{\E}{\mc E}
 \nc{\ba}{\tilde{\pa}}
 \nc{\half}{\frac{1}{2}}
 \nc{\hgt}{\text{ht}}
 \nc{\ka}{\kappa}
 \nc{\mc}{\mathcal}
 \nc{\mf}{\mathfrak} 
 \nc{\hf}{\frac{1}{2}}
\nc{\ov}{\overline}
\nc{\ul}{\underline}
\nc{\I}{\mathbb{I}}
\nc{\xx}{{\mf x}}
\nc{\id}{\text{id}}
\nc{\one}{\bold{1}}
\nc{\mfsl}{\mf{sl}}
\nc{\mfgl}{\mf{gl}}
\nc{\ti}[1]{\textit{#1}}
\nc{\Hom}{\mathrm{Hom}}
\nc{\Irr}{\mathrm{Irr}}
\nc{\Cat}{\mathscr{C}}
\nc{\CatO}{\mathscr{O}}
\renewcommand{\O}{\mathrm{O}}
\nc{\Tan}{\mathscr{T}}
\nc{\Umod}{\mathscr{U}}
\nc{\Func}{\mathscr{F}}
\nc{\Kh}{\text{Kh}}
\nc{\Khb}[1]{\llbracket #1 \rrbracket}

\nc{\ua}{\mf{u}}
\nc{\nb}{u}
\nc{\inv}{\theta}
\nc{\mA}{\mathcal{A}}
\newcommand{\TT}{\mathbf T}
\newcommand{\TA}{{}_\A{\TT}}
\newcommand{\tK}{\widetilde{K}}
\newcommand{\al}{\alpha}
\newcommand{\Fr}{\bold{Fr}}

\nc{\Qq}{\Q(v)}
\nc{\uu}{\mathfrak{u}}
\nc{\Udot}{\dot{\U}}

\nc{\f}{\bold{f}}
\nc{\fprime}{\bold{'f}}
\nc{\B}{\bold{B}}
\nc{\Bdot}{\dot{\B}}
\nc{\Dupsilon}{\Upsilon^{\vartriangle}}
\newcommand{\T}{\texttt T}
\newcommand{\vs}{\varsigma}
\newcommand{\Pa}{{\bf{P}}}
\newcommand{\Padot}{\dot{\bf{P}}}

\nc{\ipsi}{\psi_{\imath}}
\nc{\Ui}{{\bold{U}^{\imath}}}
\nc{\uidot}{\dot{\mathfrak{u}}^{\imath}}
\nc{\Uidot}{\dot{\bold{U}}^{\imath}}
 \nc{\be}{e}
 \nc{\bff}{f}
 \nc{\bk}{k}
 \nc{\bt}{t}
 \nc{\bs}{\backslash}
 \nc{\BLambda}{{\Lambda_{\inv}}}
\nc{\Ktilde}{\widetilde{K}}
\nc{\bktilde}{\widetilde{k}}
\nc{\Yi}{Y^{w_0}}
\nc{\bunlambda}{\Lambda^\imath}
\newcommand{\Iwhite}{\I_{\circ}}
\nc{\ile}{\le_\imath}
\nc{\il}{<_{\imath}}

\newcommand{\ff}{B}


\nc{\etab}{\eta^{\bullet}}
\newcommand{\Iblack}{\I_{\bullet}}
\newcommand{\wb}{w_\bullet}
\newcommand{\UIblack}{\U_{\Iblack}}

\newcommand{\blue}[1]{{\color{blue}#1}}
\newcommand{\red}[1]{{\color{red}#1}}
\newcommand{\green}[1]{{\color{green}#1}}
\newcommand{\white}[1]{{\color{white}#1}}

\newcommand{\dvd}[1]{t_{\odd}^{{(#1)}}}
\newcommand{\dvp}[1]{t_{\ev}^{{(#1)}}}
\newcommand{\ev}{\mathrm{ev}}
\newcommand{\odd}{\mathrm{odd}}

\newcommand\TikCircle[1][2.5]{{\mathop{\tikz[baseline=-#1]{\draw[thick](0,0)circle[radius=#1mm];}}}}

\newcommand{\commentcustom}[1]{}

\raggedbottom

\title[Global theta and periods]
{Global theta lifting and automorphic periods associated to nilpotent orbits}

\author{Bryan Wang Peng Jun}
\address{Department of Mathematics, National University of Singapore, 10 Lower Kent Ridge Road, Singapore 119076}
\email{bwang@u.nus.edu}

\begin{abstract}
A systematic way to organise the interesting periods of automorphic forms on a reductive group $G$ is via the theory of nilpotent orbits of $G$. On the other hand, it is known that the theta correspondence can be used effectively to relate automorphic periods on each member of a dual pair. In this paper, we establish this relation in full generality, facilitated by a certain transfer of nilpotent orbits via moment maps. This is the analogous global result to the local result previously established by Gomez and Zhu. 
\end{abstract}

\maketitle

\setcounter{tocdepth}{1}
\tableofcontents

\section{Introduction}\label{sec:Introduction}

Period integrals have long been of key interest in the study of automorphic forms, and more specifically also in the Langlands program. For instance, one of the central problems in the relative Langlands program is to characterize the non-vanishing of certain periods of automorphic forms as the image of certain Langlands functorial liftings, or by analytic properties of automorphic $L$-functions.

A systematic way to organise the possible periods that one is interested in is via the (Dynkin-Kostant) theory of nilpotent orbits for a reductive group $G$. For an introduction, one may refer to the notes of Jiang \cite{Ji}. Much more recently, the work of Ben-Zvi, Sakellaridis and Venkatesh on the relative Langlands program \cite{BZSV} has led to a far-reaching framework for the study of periods of automorphic forms, indexed by nilpotent orbits of $G$ (among other data).

One of the main tools used to approach such problems involving automorphic periods is the theory of theta correspondence (the other being the relative trace formula). It is a known principle that theta correspondence often relates two periods on each member of a reductive dual pair. As such, the theta correspondence can often be used to great effect to establish instances of the characterization of non-vanishing of periods. 

A natural question which arises is the following. Suppose $(G,G')$ is a reductive dual pair, and theta correspondence relates an automorphic period on $G$ corresponding to a nilpotent orbit $\gamma$ of $G$, with an automorphic period on $G'$ corresponding to a nilpotent orbit $\gamma'$ of $G'$. Then can one understand the possible relations between $\gamma$ and $\gamma'$, so that one has a better understanding of the possible periods that can be related by the theta correspondence?

This was studied in the analogous local setting by Gomez and Zhu \cite{GZ}, where they showed that $\gamma$ and $\gamma'$ can be understood as being related by a certain natural transfer of nilpotent orbits via moment maps (Section \ref{NilpOrbitTransfer}). It was further shown in \cite{GW} that this moment map transfer has natural interpretations in the \cite{BZSV}-framework indexing period problems by Hamiltonian spaces (which come equipped with moment maps). 

\vskip 5pt

The aim of this paper is to establish the analogous result of \cite{GZ} (and its generalisation as given in \cite{Zh}), in the global setting. Suppose $\gamma$ and $\gamma'$ are related by the moment map transfer of nilpotent orbits (as in Section \ref{NilpOrbitTransfer}). Then we explicate the relation between the $\gamma$-period on a (cuspidal) automorphic representation $\pi$ of $G$, and the $\gamma'$-period of its global theta lifting $\Theta(\pi)$ of $G'$ (Theorem \ref{thm:KeyResult}). This allows one to characterize the non-vanishing of the $\gamma$-period on $\pi$ in terms of the non-vanishing of the $\gamma'$-period on $\Theta(\pi)$ (and vice versa). 

We leave the formulation of the precise statements to the main body of the paper, with the main results summarised in Theorems \ref{KeyTheoremNotInImage}, \ref{KeyTheorem} and \ref{KeyTheoremConverse}. The paper is organised as follows. We set up the necessary notation and preliminaries in Section \ref{sec:Notation}. We then compute the period integral of the relevant theta function in Section \ref{sec:KeyComputation}, which is the heart of the computation. We complete the computation and obtain our main results in Section \ref{sec:RelatePeriods}. Finally, we illustrate some examples in Section \ref{sec:Examples}.

\vskip 5pt

Let us finally remark that such global computations of periods of global theta lifts have been done multiple times in the literature, all in various special cases, and it would be impossible to give an exhaustive list of references here. We mention here only three such examples (again, there are many others). 

\begin{itemize}
    \item One is in \cite{GRS}, where the case corresponding, in the notation of this paper, to $(G,G')=(\Sp_{2n},\O_{2n})$, $\gamma$ with partition $[2n-2,1,1]$, and $\gamma'$ with partition $[2n-1,1]$ (the regular orbit), was considered (as well as several other similar cases).
    \item Another is in \cite[Section 8]{Xu}, where the case with $(G,G')=(\O_{2n+2},\Sp_{2n+2})$, $\gamma$ with partition $[1^{2n+2}]$ (the trivial orbit), and $\gamma'$ with partition $[2,1^{2n}]$ was considered.
    \item Yet another example still is in \cite{BM}, where essentially the case with $(G,G')=(\O_{m+n}, \Sp_{2m})$, $\gamma$ with partition $[1^{m+n}]$ (the trivial orbit), and $\gamma'$ with partition $[2^m]$\footnote{Note that there is a difference in terminology, as also noted in \cite{BM}; in \cite{BM}, this case is referred to as a Bessel period, whereas we would refer to this case as a Shalika period.} was considered. 
\end{itemize}  

The main contribution of this paper is therefore to show how such a global computation can be performed in essentially full generality, as facilitated by the moment map transfer of nilpotent orbits. 

\subsection{Acknowledgements} The author thanks his undergraduate advisor Wee Teck Gan, not only for suggesting this problem and for many helpful comments on the paper, but also for his constant guidance and support. The author also thanks Chengbo Zhu for his helpful comments on the paper, and Nhat Hoang Le for several helpful discussions on related topics. 

\vskip 5pt

\section{Notation and preliminaries}\label{sec:Notation}

Let $k$ be a number field with ring of adeles $\A$, and fix a non-trivial unitary additive character $\psi = \otimes_v \psi_v : k\bs \A \rightarrow \C^\times$. In particular $\overline{\psi}=\psi^{-1}$. Since $\psi$ is fixed, in what follows, wherever $\psi$ is omitted from the notation, then we understand that we are working with this fixed $\psi$, so as to avoid excessive notation. For any algebraic group $G$ over $k$, we denote as usual $[G]:= G(k)\bs G(\A)$. All measures on adelic groups are taken to be the usual Tamagawa measure. 

\subsection{Reductive dual pairs and global theta lift}\label{DualPairs}

Let $(G',G)$ be a type I reductive dual pair; in this paper, we will take $G',G$ to be the isometry groups of an orthogonal vector space $(V',B')$ and a symplectic vector space $(V,B)$ (or vice versa)\footnote{One may also work with the unitary or quaternionic dual pairs associated to a pair of $\epsilon$-Hermitian vector spaces over a quadratic extension or quaternion algebra over $k$, with only technical changes needed to the arguments that follow. To streamline the exposition, we focus in this paper on the case of symplectic-orthogonal dual pairs.}. If $\dim V'$ is odd, then in what follows we understand that we will have to work with the metaplectic group ${\rm Mp}(V)$ instead of $\Sp(V)$ (and similar for all other reductive dual pairs which occur in this paper, to avoid excessive notation). 

From the character $\psi$, one has the global Weil representation \[ \omega_\psi = \otimes'_v \omega_{\psi_v} \] which is realised on the space of Schwarz functions \[ \S(Y_\A) = \otimes'_v \S(Y_v), \] for $Y$ a Lagrangian subspace of the symplectic vector space $\Hom(V,V')$. 

\begin{rem}\label{rem:DualVectorSpace}

It is perhaps more common to work with the symplectic vector space $V\otimes V'$ when dealing with the Weil representation. However, we may always identify $V\otimes V'$ with $\Hom(V,V')$ via the form $B$. 

    To be precise, in what follows, we will always implicitly identify $V$ with its linear dual $V^\ast$ using the form $B$ (via $v\mapsto B(v,\cdot)$) without further comment. Thus, for instance, given a linear map $f\in \Hom(V,V')$, its adjoint $f^\ast$ is an element of $\Hom(V',V)$. 

\end{rem}

As usual, for $\phi\in \S(Y_\A)$, one has the formation of theta series \[ \theta(\phi)(g) = \sum_{y\in Y_k} (\omega_\psi(g)\phi)(y). \] 

For a cuspidal automorphic representation $\pi$ of $G$, one has the global theta lift $\Theta(\pi)$ of $G'$, which is spanned by \[ \theta(\phi, f)(g') := \int_{[G]} \theta(\phi)(gg')\overline{f(g)} \mathop{dg} \] for $\phi\in \S(Y_\A)$  and $f\in \pi$. 

\subsection{Nilpotent orbits}\label{NilpOrbits} 

In this section, we set up the notation and preliminaries concerning either only the group $G$ or the group $G'$. In the subsequent sections, we consider the notation and preliminaries which involve both $G$ and $G'$.

In the following we will focus on the group $G$; similar definitions are made for the group $G'$ and notated subsequently by the occurrence of $^\prime$. In this paper, we will enforce religiously the use of $^\prime$ for all objects and variables related to the group $G'$.

Work first over the number field $k$ or one of its completions $k_v$. Let $\gamma=\{e,h,f\}\subset \g$ be an $\mfsl_{2}$-triple associated to a nilpotent orbit of $\g$. 

\begin{rem}
By the Jacobson-Morozov theorem (and other results of Kostant) \cite{CM}, there is a correspondence between (conjugacy classes of) $\mfsl_2$-triples and nilpotent orbits; as such, in this paper, we will often refer to $\mfsl_2$-triples and their corresponding nilpotent orbits interchangeably, with no confusion to be expected for the reader. 
\end{rem}

Under the adjoint $\mfsl_2$-action, $\g$ decomposes into $\mfsl_2$ weightspaces \[\g_j=\{v\in \g \mid {\rm ad}(h)v=jv\}\] for $j\in \Z$. We have the parabolic \[\mf p=\oplus_{j\le 0} \g_{j} = \mf m \oplus \mf u,\] where $\mf m = \g_0$. Set $\mf u^+:=\oplus_{j\leq -2} \g_{j}$.

We get corresponding subgroups $P=M\ltimes U$ and $U^+$  of $G$. Note that \[ M=\{m\in G \mid \mbox{${\rm Ad}(m)h=h$}\}\] is the centraliser of $h$. Denote the centraliser of $\gamma$ by \[M_\gamma=\{m\in M \mid \mbox{${\rm Ad}(m)e=e$}\}=\{g\in G \mid \mbox{${\rm Ad}(g)e=e$, ${\rm Ad}(g)f=f$, ${\rm Ad}(g)h=h$}\},\] which is reductive. 

Fix $\kappa$, an ${\rm Ad}(G)$-invariant non-degenerate bilinear form on $\g$. For each place $v$ of $k$, from the character $\psi_v$, we have a character $\chi_{\gamma,\psi_v}$ on $U^+$ via 
\begin{equation}
\label{defchi}
\chi_{\gamma,\psi_v}(\exp u):=\psi_v (\kappa(e,u)), \ \ \ \ \forall \ u\in \mf u^+.
\end{equation}
Denote also \[ \kappa_e(u) := \kappa(e,u). \] 

\vskip 5pt

In the adelic setting, we thus have also a character $\chi_{\gamma,\psi}$ on $U^+(k) \bs U^+(\A)$, beginning with a nilpotent orbit $\gamma$ over $k$.  

\vskip 5pt

From now on, suppose $G=G(V)$ is a classical group that we are dealing with in this paper, as in the previous Section \ref{DualPairs}. Let us fix (the normalisation of) $\kappa_e$ as being equal to $u\mapsto \frac{1}{2}\Tr(eu)$. Here we identify $\g$ as a subalgebra of $\End(V)$, and the trace is taken as linear maps from $V$ to itself.

\subsubsection{Classification of nilpotent orbits for classical groups}\label{NilpPartition} 

Continue to suppose $G=G(V)$. 

From an $\mfsl_2$-triple as above, we obtain an $\mfsl_2$-representation on $V$ and the decomposition \[ V=\oplus_{j=1}^l W_j^{\oplus a_j}, \] where \[ W_j^{\oplus a_j} \cong W_j \otimes V^j \] is the isotypic component of $V$ for the irreducible $j$-dimensional representation $W_j$ of $\mfsl_2$, and $V^j=\Hom(W_j,V)$ is a $a_j$-dimensional multiplicity space. (We use the superscript ${}^j$ to emphasise that this is a multiplicity space, and to distinguish it from $V_j$, which will occur later.)

Recall, from standard $\mfsl_2$-theory, that $W_j$ is symplectic (resp. orthogonal) if $j$ is even (resp. odd); fix corresponding $\mfsl_2$-invariant forms $A_j$ on $W_j$.

The form $B$ induces a symplectic or orthogonal form $B_j$ on the $a_j$-dimensional multiplicity spaces $V^j$. Moreover, $B_j$ is symplectic if $j$ is even and $B$ is orthogonal or if $j$ is odd and $B$ is symplectic, and otherwise $B_j$ is orthogonal. 

The centraliser $M_\gamma$ is in fact (canonically isomorphic to) the direct product of the isometry groups of $(V^j,B_j)$; we denote \[ M_\gamma \cong \prod_{j=1}^l G(V^j,B_j). \] 

The above furnishes a parameterisation of the nilpotent orbits in $G$, by the datum of:
\begin{itemize}
    \item the partition $\lambda=[l^{a_l},\dots, 1^{a_1}]$ of $n$, and
    \item the forms on the multiplicity spaces $(V^j,B_j)$,
\end{itemize} 
such that the $(V^j,B_j)$ are compatible with $B$ in the above way; to be precise, this means that the $B_j$ must be the forms that would be induced from $B$ as above, or that \[ \bigoplus_j (V^j, B_j) \otimes (W_j, A_j) \cong (V, B). \]

In particular if $G$ is an orthogonal group, then even parts must occur with even multiplicity in $\lambda$, and if $G$ is symplectic, then odd parts must occur with even multiplicity in $\lambda$. 

\subsubsection{Weight space decomposition}\label{WeightSpace}

Continuing with the $\mfsl_2$-representation on $V$, we have the $\mfsl_2$-weight space decomposition \[ V = V_{-r}\oplus\cdots\oplus V_{r}. \] 

Observe that by $\mfsl_2$-theory, $B$ restricts to a perfect pairing between $V_m$ and $V_{-m}$ for each $1\le m\le r$. Hence, for instance, as in Remark \ref{rem:DualVectorSpace}, we will identify $V_m^\ast = V_{-m}$. Moreover, for each $k\ne 0$, $V_k$ is an isotropic subspace. 

Finally, note that if we define a flag $(\overline{V}_t)$ of subspaces of $V$ by \[ \overline{V}_t := \bigoplus_{j=-r}^t V_j, \] then the parabolic $P=MU$ is in fact the stabiliser of this flag. 

In what follows, we will use brackets in the subscript to denote, for example, for $0\le m\le r$, \[ V_{(m)} := V_{-m}\oplus\cdots\oplus V_m, \] and $G_{(m)} = G(V_{(m)})$, the isometry group of $V_{(m)}$.

For $1\le m\le r$, denote by $P_m$ the stabiliser of $V_{-m}$ in $G_{(m)}$; it is a parabolic subgroup of $G_{(m)}$ with $P_m = M_{(m)} U_m$, where $M_{(m)} = M_m \times G_{(m-1)}$ with $M_m\cong \GL(V_{-m})$. If $U_{(m)} := U \cap G_{(m)}$, then one has \[ U_{(m)}=\big( (U_1\ltimes U_2)\ltimes \cdots \big) \ltimes U_m. \]

Denote $U^{+}_m = U^{+} \cap U_m$. Similarly, if $U^{+}_{(m)} := U^{+} \cap G_{(m)}$, then one has $U^{+}_{(m)} = U^{+}_1 \cdots U^{+}_{m}$. 

One may define $\chi_{\gamma,m} = \chi_\gamma |_{U^{+}_m}$, and then clearly $\chi_\gamma = \chi_{\gamma,1}\cdots \chi_{\gamma,r}$.

\begin{figure}[!h]
\begin{center}
\renewcommand{\arraystretch}{6}
\[ \left(\begin{array}{@{}c|c|c@{}}
  ^t m_r^{-1}
  & \bigzero & \bigzero  \\
\hline
  -q^\ast &
  G_{(r-1)} &\bigzero \\  \hline
  z\in \Hom(V_r, V_{-r}) & q\in \Hom(V_{(r-1)},V_{-r}) & m_r \in \GL(V_{-r})
\end{array}\right) \]
\end{center} 
\vspace{1em}
\caption{A matrix representation of elements of $G=G_{(r)}$, with a chosen ordered basis of $V$ corresponding to the weight space decomposition, in order of decreasing weight. }\label{fig:MatrixRep}
\end{figure}

For each $m$, the Lie algebra $\mf{u}_m$ has a vector space decomposition \[ \mf{u}_{m} \cong \mf{q}_{m} \oplus \mf{z}_{m} \cong \Hom(V_{(m-1)},V_{-m}) \oplus \mf{z}_{m}.  \] Here $\mf{z}_{m}$ is a Lie subalgebra and is central in $\mf{u}_m$, and  \[ \mf{z}_m = \{ z\in \Hom(V_m, V_{-m}) | z^\ast = -z\} \subseteq \g_{-2m}, \] while the isomorphism $ \Hom(V_{(m-1)},V_{-m})\cong \mf{q}_{m} $ is provided by \[ q\mapsto q-q^\ast.  \]

One has a corresponding short exact sequence \[ 1\rightarrow Z_m \rightarrow U_m \rightarrow Q_m \rightarrow 1.\]

\vskip 5pt

We have a similar decomposition for $U^+_m = U^+ \cap U_m$. Write similarly $Z^+_m = U^+ \cap Z_m$. Since we are now considering those elements of $\g$ with (absolute value of) weight at least 2 under the adjoint $\mfsl_2$-action, we may identify \[ \mf{q}^+_{m} \cong  \Hom(V_{(m-2)}\oplus V_{m-1},V_{-m}). \] 

In particular for $m\ge 1$, observe that one has $Z_m = Z^+_m$. 

See Figure \ref{fig:MatrixRep} for an illustration of a matrix representation of elements of $G=G_{(r)}$, which may help with visualisation of the definitions made here.

\vskip 5pt

One makes similar definitions for the group $G'$, except that $P'_m$ is the stabiliser of $V'_{m}$ in $G'_{(m)}$, and we have the identifications (again as vector spaces) \[ \mf{q}'_{m} \cong \Hom(V'_{m},V'_{(m-1)}), \] and \[  \mf{q}^{+\prime}_{m} \cong \Hom(V'_{m},V'_{(m-2)}\oplus V'_{-(m-1)}). \] 

The reason for the apparent difference in considering the stabiliser of $V'_m$ (rather than $V'_{-m}$) is due to the identification of $V$ and $V^\ast$ that we have made in Remark \ref{rem:DualVectorSpace}.

\subsection{Moment map transfer of nilpotent orbits}\label{NilpOrbitTransfer}

We work now over the number field $k$. 

\begin{prop}\label{prop:TransferNilpOrbit} (Transfer of nilpotent orbits via moment map) \cite[Section 3.2]{Zh}

One has moment maps \[  \begin{tikzcd}
 & &\Hom(V,V')\arrow{ld}{}[swap]{\phi'} \arrow[rd, "\phi"] & \\
& \g' & & \g 
\end{tikzcd} \]
defined by \[\phi(f) = f^\ast f,\] \[\phi'(f) = f f^\ast.\] 

Given a nilpotent element $e'$ in the image of $\phi'$ corresponding to a $\mfsl_2$-triple $\gamma'$, one may uniquely define a nilpotent orbit/conjugacy class of $\mfsl_2$-triple $\gamma$ of $\g$ (with corresponding nilpotent element $e\in\g$) such that:
\begin{itemize}
    \item $e,e'$ are the images of some common element $f\in\Hom(V,V')$;
    \item the form on $V$ restricts to a nondegenerate form on $\ker f$ (including if $\ker f = 0$),
    \item and $f$ sends $V_k$, the $k$-weight space of $V$, to $V'_{k+1}$, the $(k+1)$-weight space of $V'$ for all $k\in \Z$ (here the weight spaces are under the $\mfsl_2$ action coming from $\gamma,\gamma'$, as in section \ref{WeightSpace}). 
\end{itemize}

The second and third conditions may be summarised in the following diagram:\[ 
\begin{array}{c}   V_{-r}   \oplus    V_{-(r-1)}    \oplus    \cdots    \oplus    V_{r-1}   \oplus   V_{r}   \\
  \hspace{45pt} \searrow \hspace{-3pt} {\scriptstyle f_{-r}}  \hspace{10pt}  \searrow \hspace{-3pt} {\scriptstyle f_{-(r-1)}} \hspace{5pt}  \cdots\hspace{5pt}    \searrow \hspace{-3pt} {\scriptstyle f_{r-1}} \hspace{-5pt}   \searrow \hspace{-3pt} {\scriptstyle f_{r}}    \\
{V'}_{-(r+1)}    \oplus     {V'}_{-r}     \oplus     {V'}_{-(r-1)}    \oplus    \cdots    \oplus    {V'}_{r-1}   \oplus   {V'}_{r}    \oplus   {V'}_{r+1}
\end{array}
\] 

where here $f_m = f|_{V_m}$ for each $m$, and $\ker f=\ker f_0$, $\ker f_m = 0$ for $m\ne 0$. One can show \cite[Lemma 5.6]{GZ} that $f_m$ must be an isomorphism for $m>0$.

\end{prop}

In what follows we will implicitly assume, unless otherwise specified, that we are working with $e'\in \gamma'$ which is in the image of the moment map $\phi'$, so that the condition of the preceding proposition is satisfied. 

The partitions corresponding to $\gamma',\gamma$ are related in the following way: suppose their corresponding Young tableaux are $\lambda',\lambda$ respectively. Then one removes the first column of $\lambda'$ and adds suitably many rows of length 1, to obtain $\lambda$. 

Recall from Section \ref{NilpPartition} that the nilpotent orbits of $\g',\g$ are parameterised also by (symplectic or orthogonal) forms $B'_j, B_j$ on the multiplicity spaces $V^{j\prime}, V^j$. To obtain the corresponding forms $B_j$ for $\gamma$, the forms $B'_j$ from $\gamma'$ are left unchanged, (and the form $B_1$ corresponding to the rows of length 1 in $\lambda$ is determined by the compatibility condition of Section \ref{NilpPartition}). 

In other words, one has \[ (V^j,B_j)  \cong (V^{(j+1)\prime},B'_{j+1}) \quad \text{for $j \geq 2$} \]
and \begin{equation}\label{Vnew} V^1 \cong  V^{2\prime} \oplus V_{new} \end{equation} for a nondegenerate subspace $V_{new}$ corresponding to the \textit{newly added} rows of length 1 in $\lambda$. In fact, $V_{new}=\ker f$.

See Figure \ref{fig:TransferNilpOrbits} for an illustration. 

\begin{figure}[!h]
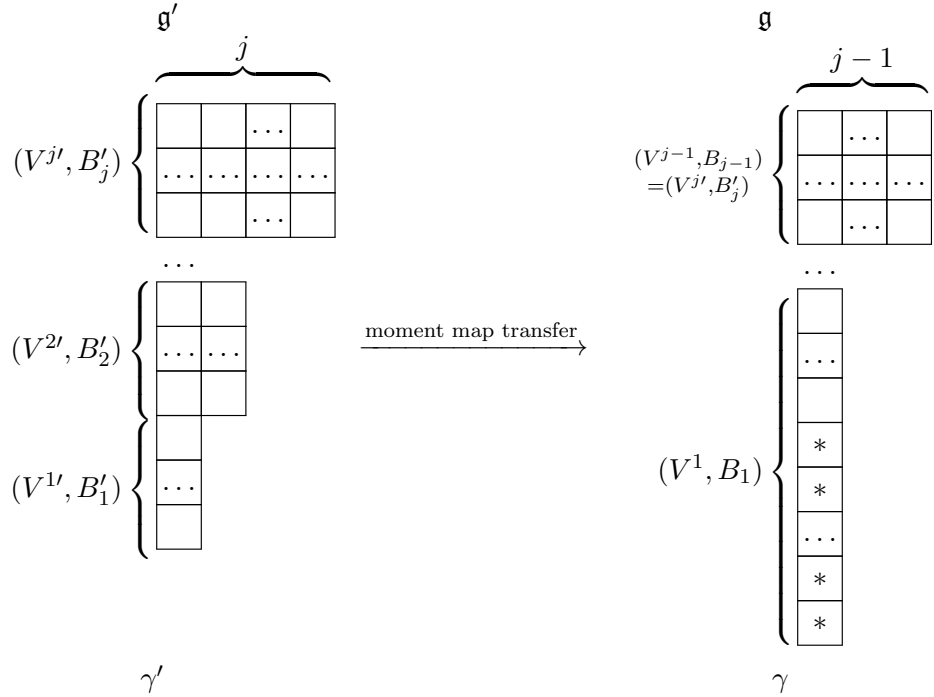

    \begin{center}
$\mf{g}' \quad\quad\quad\quad\quad\quad\quad\quad\quad\quad\quad\quad\quad\quad\quad\quad\quad\quad\quad\quad \mf{g}$
    
   \ytableausetup {mathmode,boxframe=normal,boxsize=1.5em,centertableaux} 
   \begin{tabular}{r@{}l}
   & $\overbrace{\hspace{6em}}^{\displaystyle j}$\\
   \begin{tabular}{r@{}l} $(V^{j\prime}, B'_j)\left\{\vphantom{\begin{ytableau}
    \quad \\ \quad  \\ \quad
 \end{ytableau}}\right.$ \\ \vspace{0.5em} \\
 $(V^{2\prime}, B'_2)\left\{\vphantom{\begin{ytableau}
    \quad \\ \quad  \\ \quad
 \end{ytableau}}\right.$ \\ $(V^{1\prime}, B'_1)\left\{\vphantom{\begin{ytableau}
    \quad \\ \quad  \\ \quad
 \end{ytableau}}\right.$ \end{tabular} & \begin{ytableau} \quad & \quad &\dots &\quad \\ \dots & \dots &\dots &\dots \\ \quad & \quad &\dots &\quad \\ \none[\dots] \\ \quad & \quad \\ \dots &\dots \\ \quad &\quad \\ \quad \\ \dots \\ \quad  \end{ytableau} \\ \\ \\ \\
   \end{tabular} $\xrightarrow{\text{moment map transfer}}$
\ytableausetup {mathmode,boxframe=normal,boxsize=1.5em,centertableaux} 
   \begin{tabular}{r@{}l} 
   & $\overbrace{\hspace{4.5em}}^{\displaystyle j-1}$\\
   \begin{tabular}{r@{}l} $\substack{(V^{j-1}, B_{j-1}) \\ = (V^{j\prime}, B'_j)}  \left\{\vphantom{\begin{ytableau}
    \quad \\ \quad  \\ \quad
 \end{ytableau}}\right.$ \\ \vspace{0.5em} \\
 $(V^1, B_1)\left\{\vphantom{\begin{ytableau}
    \quad \\ \quad  \\ \quad \\ \quad \\ \quad  \\ \quad \\ \quad \\ \quad
 \end{ytableau}}\right.$  \end{tabular} & \begin{ytableau} \quad &\dots &\quad \\  \dots &\dots &\dots \\  \quad &\dots &\quad \\ \none[\dots] \\   \quad \\ \dots \\ \quad \\ \ast \\ \ast \\ \dots \\ \ast \\ \ast  \end{ytableau}
   \end{tabular}

$\gamma' \quad\quad\quad\quad\quad\quad\quad\quad\quad\quad\quad\quad\quad\quad\quad\quad\quad\quad\quad\quad\quad \gamma$

\end{center}
\vspace{1em}
\caption{An illustration of the transfer of nilpotent orbits via the moment maps, in terms of the Young tableaux $\lambda'$ and $\lambda$. \\ $\ast$ indicates the newly added rows of length 1 in $\lambda$. }\label{fig:TransferNilpOrbits}
\end{figure}

Recall from Section \ref{NilpPartition} that \[ M_{\gamma'}  \cong \prod_{k=1}^j G'(V^{k\prime},B'_k) \] and \[ M_{\gamma}  \cong \prod_{k=1}^j G(V^k,B_k). \]

In particular $M_{\gamma'}$ and $M_{\gamma}$ contain respectively factors 
\[G'(V^{1\prime},B'_1)\] and \[G(V^1,B_1), \] corresponding respectively to the rows of length 1 in $\lambda'$ and $\lambda$. 

Furthermore $G(V^1,B_1)$ contains a subgroup $G(V_{new})$ (cf. (\ref{Vnew})), which is an isometry group of the subspace $V_{new}\subseteq V^1$ corresponding to the \textit{newly added} rows of length 1 in $\lambda$ (cf. Proposition \ref{prop:TransferNilpOrbit}).

We have that $G'(V^{1\prime},B'_1)$ and $G(V_{new})$ forms a reductive dual pair inside \[ \Sp\big(\Hom(V_{new}, V^{1\prime})\big), \] and of the same type as $G'$ and $G$ respectively. Note here that $\Hom(V_{new}, V^{1\prime})$ is a subspace of $\Hom(V,V')$ with the inherited symplectic form. 

Write $(L',L)$ for this dual pair, and $\omega_{\gamma,\gamma'}$ for the Weil representation associated to this dual pair. 

\subsection{Weil representation and Lagrangians}\label{WeilRep}

Write $\omega_{(l),(m)}$ for the Weil representation associated to the dual pair $(G_{(l)}, G'_{(m)})$; it is realised on the space of Schwarz functions on a Lagrangian $Y_{(l),(m)}$ of the symplectic vector space $\Hom(V_{(l)},V'_{(m)})$. 

Now let us choose Lagrangian subspaces once and for all in the following way. Choose a Lagrangian $Y_{(0),(0)}$ of $\Hom(V_{(0)},V'_{(0)})$. Then for each $0\le j\le r$, set \[ Y_{(j),(j+1)} = \Hom(V_{(j)}, V'_{j+1}) \oplus Y_{(j),(j)} \] and for each $1\le j\le r$, set \[ Y_{(j),(j)} = \Hom(V_{-j}, V'_{(j)}) \oplus Y_{(j-1),(j)}. \]

It is clear that the Weil representation $\omega_{(r),(r+1)}$ may be viewed as realised on the space of Schwarz functions on $\Hom(V_{(r)}, V'_{r+1})$ taking values in $\omega_{(r),(r)}$. We record here some pertinent formulas for the Weil representation $\omega_{(r),(r+1)}$, to be used later in the global computation. 

For the action of $G_{(r)}$, one has 
\begin{equation}\label{Graction}
    (\omega_{(r),(r+1)}(g)\phi)(F) = \omega_{(r),(r)}(g) \phi(Fg)
\end{equation} 
for $g\in G_{(r)}$ and $F\in \Hom(V_{(r)}, V'_{r+1})$. 

For the action of $U'_{r+1}$, recall from section \ref{WeightSpace} that one has the vector space decomposition  \[ \mf{u}'_{r+1} \cong \mf{q}'_{r+1} \oplus \mf{z}'_{r+1} \cong \Hom(V'_{r+1},V'_{(r)}) \oplus \mf{z}'_{r+1}  \] and $\mf{z}'_{r+1}\subseteq \Hom(V'_{r+1}, V'_{-(r+1)})$.

Then one has \begin{equation}\label{Zr+1action}
    (\omega_{(r),(r+1)}(\exp z')\phi)(F) = \psi(\mathop{\Tr} z'FF^\ast/2) \phi(F)
\end{equation} 
for $z'\in \mf{z}'_{r+1}$ and $F\in \Hom(V_{(r)}, V'_{r+1})$, and
\begin{equation}\label{Rr+1action}
    (\omega_{(r),(r+1)}(\exp q')\phi)(F) = \omega_{(r),(r)}(q'F) \phi(F)
\end{equation}
for $q'\in \mf{q}'_{r+1} \cong \Hom(V'_{r+1},V'_{(r)})$ and $F\in \Hom(V_{(r)}, V'_{r+1})$, and here $\omega_{(r),(r)}(q'F)$ gives the action of the Heisenberg group associated to $\Hom(V_{(r)},V'_{(r)})$. 

One also has \begin{equation}\label{Mr+1action}
    (\omega_{(r),(r+1)}(m')\phi)(F) = \nu(m') \phi(m'^{-1} F)
\end{equation} 
for $m'\in M'_{r+1}\cong \GL(V'_{r+1})$ and some character $\nu$ whose explicit form is not important. 

\subsection{Periods}\label{Periods}

Observe that $U'/U^{+\prime}\cong \g'_{-1}$ may be given the structure of a symplectic vector space, coming from the non-degenerate pairing $\kappa'$ between $\g'_{-1}$ and $\g'_1$, via
\[ \langle v,w\rangle := \kappa'({\rm ad}(e)v, w) ,\]
and this symplectic form is $M_{\gamma'}$-invariant (hence, in particular, $L'$-invariant). 

We thus obtain the Heisenberg-Weil representation $\omega_{\gamma'}$ of $U'$, on which $U^{+\prime}$ acts by the character $\chi_{\gamma'}$. It is realised on the space of Schwarz functions on a Lagrangian of the symplectic vector space $\g'_{-1}$. Denote by $\theta_{\gamma'}$ the corresponding formation of theta series.  

\begin{definition}
    For $f'$ a cusp form on $G'$, $\phi_{\gamma'}\in \omega_{\gamma'}$, and $\tau'$ a cusp form on $L'$, define the period integral \[ P_{\gamma',\psi}(f', \phi_{\gamma'}, \tau') = \int_{[L']} \int_{[U']} f'(u'l') \cdot \overline{ \theta_{\gamma'}(\phi_{\gamma'})(u'l') \cdot \tau'(l')  } \mathop{du'} \mathop{dl'} \]

We also make the corresponding definition for $G$ (with $L,U$ in place of $L',U'$). 
\end{definition} 

By \cite[Proposition A.1.1(ix)]{BP}, this period integral is absolutely convergent even without any cuspidality assumption on $\tau'$, thanks to the cuspidality of $f'$. 

\section{Key computation}\label{sec:KeyComputation}

 Let $\pi$ be a cuspidal automorphic representation of $G$, and $\tau'$ a cusp form on $L'$. For a cusp form $f\in \pi$, and $\phi\in \omega$, we would like to compute the period integral of the theta lift to $G'$, 
 
\begin{align*}
& P_{\gamma',\psi}(\theta(\phi,f), \phi_{\gamma'}, \tau') \\
&= \int_{[L']}  \int_{[U']}  \int_{[G]} \overline{f(g)} \cdot\theta(\phi)(gu'l') \cdot\overline{ \theta_{\gamma'}(\phi_{\gamma'})(u'l') \cdot \tau'(l')  } \mathop{dg} \mathop{du'} \mathop{dl'} \\
&= \int_{[G]} \overline{f(g)} \int_{[L']}  \int_{[U']}  \theta(\phi)(gu'l') \cdot\overline{\theta_{\gamma'}(\phi_{\gamma'})(u'l') \cdot\tau'(l') } \mathop{du'} \mathop{dl'} \mathop{dg} \\
&= \int_{[G]} \overline{f(g)} \int_{[L']}  \int_{[U'/U^{+\prime}]}  \int_{[U^{+\prime}]} \theta(\phi)(g u^{+\prime}\overline{u'}l')\cdot \overline{\chi_{\gamma'}(u^{+\prime})}\mathop{du^{+\prime} } \overline{\theta_{\gamma'}(\phi_{\gamma'})(\overline{u'}l') \cdot\tau'(l') } \mathop{d\overline{u'}} \mathop{dl'} \mathop{dg} 
\end{align*}
where $\overline{u'}$ denotes the image of $u'\in U'$ in $U'/U^{+\prime}$, and where we have interchanged the order of integration due to the cuspidality of the forms involved (and hence the absolute convergence of the integral).  

\subsection{Key computation}

Hence, one sees that the key step to relate the above integral to a period integral of $f$, is to compute the innermost integral, which is a period integral of the theta function 

\begin{equation}
\label{Key}
   \int_{U^{+\prime}(k)\bs U^{+\prime}(\A)} \overline{\chi_{\gamma'}(u')} \cdot\theta(\phi)(gu') \mathop{du'} \mathop{dg} 
\end{equation}

(To keep the notation manageable, we have dropped the superscript ${}^+$ from the variable $u^{+\prime}$.)

Recall from section \ref{WeightSpace} that one has \[ U^{+\prime}=\big( (U^{+\prime}_1\ltimes U^{+\prime}_2)\ltimes \cdots \big) \ltimes U^{+\prime}_{r+1}, \] and $\chi_{\gamma'} = \chi_{\gamma',1}\cdots \chi_{\gamma',r+1}$. The idea is to compute the integral over $U^{+\prime}$ by inductively, or iteratively, computing the integrals first over $[U^{+\prime}_{r+1}]$, and so on, down to $[U^{+\prime}_{1}]$.

\subsection{Iterative step} Therefore, we first compute 
\begin{equation}
\label{KeyInductive} 
\int_{U^{+\prime}_{r+1}(k)\bs U^{+\prime}_{r+1}(\A)} \overline{\chi_{\gamma',r+1}(u')} \cdot\theta(\phi)(gu') \mathop{du'},
\end{equation}
for $r\ge 1$. 

The idea is to substitute the definition of the theta series and perform a systematic unfolding of the integrals. 

\underline{Step 1:} As in section \ref{WeilRep}, the Weil representation $\omega=  \omega_{(r),(r+1)}$ is realised on the space of Schwarz functions on the Lagrangian $Y_{(r),(r+1)}$ of $\Hom(V_{(r)}, V'_{(r+1)})$. 

We have \[ (\ref{KeyInductive}) = \int_{U^{+\prime}_{r+1}(k)\bs U^{+\prime}_{r+1}(\A)} \overline{\chi_{\gamma', r+1}(u')} \sum_{y\in Y_{(r),(r+1)}(k)} (\omega_{(r),(r+1)}(gu')\phi)(y) \mathop{du'} \]

Furthermore recall from section \ref{WeightSpace}, that we have a short exact sequence  \[ 1\rightarrow Z^{+\prime}_{r+1} \rightarrow U^{+\prime}_{r+1} \rightarrow Q^{+\prime}_{r+1} \rightarrow 1,\] with $\chi_{\gamma',r+1}$ trivial on $Z^{+\prime}_{r+1}$ for $r\ge 1$. 

So we may write \[ (\ref{KeyInductive}) =\int_{[Q^{+\prime}_{r+1}]} \overline{\chi_{\gamma', r+1}(q')} \int_{[Z^{+\prime}_{r+1}]} \sum_{y\in Y_{(r),(r+1)}(k)} (\omega_{(r),(r+1)}(z' q' g)\phi)(y) \mathop{dz'} \mathop{dq'} \]
where, by abuse of notation, we identify elements $q'$ of $Q^{+\prime}_{r+1}$ with elements of \[ \Hom(V'_{r+1},V'_{(r)}), \] and write $\omega_{(r),(r+1)}(q')$ for the action of $\omega_{(r),(r+1)}(\exp q')$ as given in (\ref{Rr+1action}). 

Recall now from section \ref{WeilRep} the identification $Y_{(r),(r+1)} = \Hom(V_{(r)}, V'_{r+1}) \oplus Y_{(r),(r)}$, and the realisation of $\omega_{(r),(r+1)}$ on the space of Schwarz functions on $\Hom(V_{(r)}, V'_{r+1})$ taking values in $\omega_{(r),(r)}$.

We may thus write \begin{align*} (\ref{KeyInductive}) =\int_{[Q^{+\prime}_{r+1}]} \overline{\chi_{\gamma', r+1}(q')} \int_{[Z^{+\prime}_{r+1}]} \sum_{\substack{F_r\in \Hom(V_{(r)}, V'_{r+1})(k), \\ y_{(r),(r)}\in Y_{(r),(r)}(k)} } (\omega_{(r),(r+1)}(z' q' g)\phi)(F_r)(y_{(r),(r)}) \mathop{dz'} \mathop{dq'} \end{align*}

Now from (\ref{Zr+1action}), which gives the action of $Z^{+\prime}_{r+1}=Z'_{r+1}$, we see immediately that only the $F_r$ with $F_r F_r^\ast=0$ contribute to the inner integral. For those $F_r$ which do contribute, we get from (\ref{Zr+1action}) also that \begin{align*} (\ref{KeyInductive}) =\int_{[Q^{+\prime}_{r+1}]} \overline{\chi_{\gamma', r+1}(q')} \sum_{\substack{F_r\in \Hom(V_{(r)}, V'_{r+1})(k), \\ F_r F_r^\ast=0, \\ y_{(r),(r)}\in Y_{(r),(r)}(k)} } (\omega_{(r),(r+1)}( q' g)\phi)(F_r)(y_{(r),(r)}) \mathop{dq'} \end{align*}
(Note certainly that the Tamagawa number of the unipotent $Z^{+\prime}_{r+1}$ is 1.)

\underline{Step 2:} Similarly, one would now like to consider the action of $Q^{+\prime}_{r+1}$, as given by (\ref{Rr+1action}). However, the action of $Q^{+\prime}_{r+1}$ is given by the Heisenberg group action as discussed in and after (\ref{Rr+1action}), which is not so uniform to describe.

\vskip 5pt

The aim of this step is to show that only the $F_r$ with maximal rank (i.e. surjective onto $V'_{r+1}$) contribute to the sum.

Suppose therefore that $F_r$ is \textit{not} of maximal rank. By (\ref{Graction}), and noting of course that the theta series associated to $\omega_{(r),(r)}$ is invariant under $G_{(r)}(k)$, we may first without loss of generality conjugate $F_r$ under the $G_{(r)}$ action, so that $F_r$ is trivial when restricted to $V_{-r}\oplus\cdots\oplus V_{r-1}$, and we may hence view $F_r$ as an element of $\Hom(V_{r},V'_{r+1})$. (This is always possible since $\mathop{\im} F_r^\ast$ is totally isotropic and has dimension at most $\dim V'_{r+1} = \dim V_r$.)

Then observe that for $q'\in \Hom(V'_{r+1},V'_{(r)})$, one has $q'F_r\in \Hom(V_{r}, V'_{(r)})$. In other words, $q'F_r$ is in the maximal isotropic subspace of $\Hom(V_{(r)}, V'_{(r)})$ complementary to $Y_{(r),(r)} = \Hom(V_{-r}, V'_{(r)}) \oplus Y_{(r-1),(r)}$.

From (\ref{Rr+1action}), then, the action of $Q^{+\prime}_{r+1}$ is given by 
\begin{equation}\label{Rr+1action2}
    (\omega_{(r),(r+1)}(\exp q')\phi)(F_r)(F_{-r}) = \psi(\mathop{\Tr} q'F_r F_{-r}^\ast) \phi(F_r)(F_{-r})
\end{equation}
for $F_{-r}\in \Hom(V_{-r}, V'_{(r)})$. Similarly, here we have viewed the Heisenberg-Weil representation $\omega_{(r),(r)}$ as being realised on the space of Schwarz functions on $\Hom(V_{-r}, V'_{(r)})$ taking values in $\omega_{(r-1),(r)}$. 

Note now (cf. section \ref{NilpOrbits}) that the action of $\chi_{\gamma',r+1}$ is given by (the composition of $\psi$ with) the trace form $q' \mapsto \Tr(q'e')$.  

Therefore, in order for any fixed $(F_r, F_{-r})$ to contribute non-trivially to the integral, one requires \[ \Tr(q' F_r F_{-r}^\ast) = \Tr(q' e') \] for all $q' \in \mf{q}^{+\prime}_{r+1}\cong \Hom(V'_{r+1},V'_{(r-1)}\oplus V'_{-r})$, where the trace is taken for linear maps from $V'_{(r)}$ to itself. 

Since $q'$ maps trivially into $V'_r$, one readily sees that this is equivalent to \[  F_r F_{-r}^\ast|_{V'_{(r-1)}\oplus V'_{-r}} = e'|_{V'_{(r-1)}\oplus V'_{-r}},  \] as maps into $V'_{r+1}$. 

Now since, by standard $\mfsl_2$-theory, $\mathop{\im} (e'|_{V'_{r-1}})$ contains all of $V'_{r+1}$, it follows that if $F_r$ is not of maximal rank, i.e. its image does not contain all of $V'_{r+1}$, then its contribution to the integral must be zero, as desired.

\vskip 5pt

We thus obtain 
\begin{align*} (\ref{KeyInductive}) =\int_{[Q^{+\prime}_{r+1}]} \overline{\chi_{\gamma', r+1}(q')} \sum_{\substack{F_r\in \Hom(V_{(r)}, V'_{r+1})(k), \\ F_r F_r^\ast=0, \text{$F_r$ of maximal rank,} \\ y_{(r),(r)}\in Y_{(r),(r)}(k)} } (\omega_{(r),(r+1)}( q' g)\phi)(F_r)(y_{(r),(r)}) \mathop{dq'} \end{align*}

The next two steps are hence devoted to choosing representatives for the $(F_r,F_{-r})$ which contribute to the integral. 

\underline{Step 3:} Observe that the $F_r$ with $F_r F_r^\ast=0$ and of maximal rank form a single orbit under the $G=G_{(r)}$-action. Choose a representative $f_r$ in this orbit which, as in Step 2, is trivial when restricted to $V_{-r}\oplus\cdots\oplus V_{r-1}$, and restricts to an isomorphism $V_r \xrightarrow{\sim} V'_{r+1}$. The stabiliser of such $f_r$ in $G_{(r)}$ is $G_{(r-1)}U_r$. 

Considering also the action of $G_{(r)}$ as given in (\ref{Graction}), and noting of course that the theta series associated to $\omega_{(r),(r)}$ is invariant under $G_{(r)}(k)$, we may write 
\begin{align*} (\ref{KeyInductive}) = & \int_{[Q^{+\prime}_{r+1}]} \sum_{g_r\in G_{(r-1)}U_r(k)\bs G(k)} \overline{\chi_{\gamma', r+1}(q')} \\ &\sum_{ y_{(r),(r)}\in Y_{(r),(r)}(k) } (\omega_{(r),(r+1)}( q' g_r g)\phi)(f_r)(y_{(r),(r)}) \mathop{dq'} 
\end{align*}

Then, recalling the decomposition $Y_{(r),(r)} = \Hom(V_{-r}, V'_{(r)}) \oplus Y_{(r-1),(r)}$, we have

\begin{align*}
(\ref{KeyInductive}) = &\sum_{g_r\in G_{(r-1)}U_r(k)\bs G(k)} \int_{[Q^{+\prime}_{r+1}]} \overline{\chi_{\gamma', r+1}(q')} \\ &\sum_{ \substack{F_{-r}\in  \Hom(V_{-r}, V'_{(r)})(k), \\ y_{(r-1),(r)}\in Y_{(r-1),(r)}(k) }} (\omega_{(r),(r+1)}( q' g_r g)\phi)(f_r)(F_{-r})(y_{(r-1),(r)}) \mathop{dq'},
\end{align*}

Furthermore we have interchanged the summation $\sum_{g_r\in G_{(r-1)}U_r(k)\bs G(k)}$ and the integral $\int_{[Q^{+\prime}_{r+1}]}$, by absolute convergence. 

\begin{rem}
    Suppose we do not assume that $\gamma'$ is in the image of the moment map $\phi'$ (Proposition \ref{prop:TransferNilpOrbit}). Then, in order for the period integral under computation to be non-identically-vanishing, one will still require (in Steps 2 and 3) the existence of a $f_r$ and a (totally isotropic) subspace $V_r$ of $V$ such that $f_r$ restricts to an isomorphism $V_r \xrightarrow{\sim} V'_{r+1}$. With this convention in place, we continue the computation. The existence of such $f_r$ (for all $r$) and $f_{-r}$ in the next step will be equivalent to $\gamma'$ being in the image of the moment map $\phi'$, cf. the proof of Proposition \ref{KeyComputation} at the end of the section.
\end{rem}

\underline{Step 4:} As in Step 2, in order for $F_{-r}\in  \Hom(V_{-r}, V'_{(r)})$ to contribute non-trivially to the sum, one requires \[ \Tr(q' f_r F_{-r}^\ast) = \Tr(q' e') \] for all $q' \in \mf{q}^{+\prime}_{r+1}\cong \Hom(V'_{r+1},V'_{(r-1)}\oplus V'_{-r})$, where the trace is taken for linear maps from $V'_{(r)}$ to itself, and this is equivalent to  \[ f_r F_{-r}^\ast|_{V'_{(r-1)}\oplus V'_{-r}} = e'|_{V'_{(r-1)}\oplus V'_{-r}}. \]

Now since $e'$ maps trivially into $V'_{r+1}$ on all weight subspaces except $V'_{r-1}$, one sees that this is in turn equivalent to:
\begin{itemize}
    \item the existence of a $ F_{-r} = f_{-r} \in  \Hom(V_{-r}, V'_{(r)})$ with \[f^\ast_{-r} \in \Hom (V'_{r-1}, V_r)\] and \[f_r f_{-r}^\ast = e'\] as maps from $V'_{r-1}$ to $V'_{r+1}$;
    \item such choice of $f_{-r}$ is only unique up to translation by an element of $\Hom(V_{-r}, V'_{-r})$. 
\end{itemize} 

Consequently, we obtain 
\begin{align*} (\ref{KeyInductive}) = \sum_{g_r\in G_{(r-1)}U_r(k)\bs G(k)} \sum_{ \substack{S_{-r}\in  \Hom(V_{-r}, V'_{-r})(k), \\ y_{(r-1),(r)}\in Y_{(r-1),(r)}(k) }} (\omega_{(r),(r+1)}( g_r g)\phi)(f_r)(f_{-r}+S_{-r})(y_{(r-1),(r)}) 
\end{align*}
(as in Step 1, the Tamagawa number of the unipotent $Q^{+\prime}_{r+1}$ is 1.)

\subsection{Iteration}

Having now computed the integral over $[U^{+\prime}_{r+1}]$, one now iteratively or inductively proceeds to compute the integral over $[U^{+\prime}_{r}], \dots, [U^{+\prime}_{1}]$. 

Here one needs only the respective actions of $G_{(r-1)}$ and $U^{+\prime}_{(r)}$, and it is readily verified from the formulas for the Weil representation that
\begin{align}\label{Gr-1action}
    &\big(\omega_{(r),(r+1)}(g_{(r-1)})\phi\big)(f_r)(f_{-r}+S_{-r})(y_{(r-1),(r)}) \\
    = &\big(\omega_{(r-1),(r)}(g_{(r-1)})\, \phi(f_r)(f_{-r}+S_{-r})\big)(y_{(r-1),(r)}) \nonumber
\end{align} 
for all $g_{(r-1)}\in G_{(r-1)}$, and
\begin{align}\label{Uraction}
    &\big(\omega_{(r),(r+1)}(u'_{(r)})\phi\big)(f_r)(f_{-r}+S_{-r})(y_{(r-1),(r)}) \\
    = &\big(\omega_{(r-1),(r)}(u'_{(r)})\, \phi(f_r)(f_{-r}+S_{-r})\big)(y_{(r-1),(r)}) \nonumber
\end{align}
for all $u'_{(r)}\in U^{+\prime}_{(r)}$. For the second statement, see \cite[Lemma 6.7]{GZ}.

Therefore, the computation for $[U^{+\prime}_{r}], \dots, [U^{+\prime}_{1}]$ proceeds in completely the same way as that of $[U^{+\prime}_{r+1}]$. 

Note that, at each step, one may interchange, for instance, the summation \[ \sum_{g_r\in G_{(r-1)}U_r(k)\bs G(k)} \] and the next integral \[ \int_{U^{+\prime}_{r}(k)\bs U^{+\prime}_{r}(\A)}, \] by the absolute convergence properties of the entire integral.

\subsection{Base case}\label{BaseCase}

At the `base case' of $r=0$, there are certain differences (which in fact simplify the computation). 

First, $R^{+\prime}_1$ is trivial, hence only Step 1 need be considered and modified. Now $\chi_{\gamma',1}$ is no longer trivial on $Z^{+\prime}_1$, as one has $Z^{+\prime}_1\subseteq \g'_{-2}$ (cf. section \ref{WeightSpace}).

Therefore, one sees that only those $F_0$ with \[ F_0 F_0^\ast = e' \in \Hom(V'_{-1},V'_1) \] (instead of $F_0 F_0^\ast = 0$) contribute to the sum. Now $G_{(0)}$ acts transitively on the set of such $F_0$, and for the representative of $F_0$, one may choose a particular $F_0 = f_0$ with stabiliser $L$, which is the isometry group of $\ker f_0$ (cf. section \ref{NilpOrbitTransfer}).

\subsection{Final result of key computation} The end result of this iterative computation is therefore

\begin{align*} 
(\ref{Key}) &= \int_{U^{+\prime}(k)\bs U^{+\prime}(\A)} \overline{\chi_{\gamma'}(u')} \cdot\theta(\phi)(gu') \mathop{du'} \mathop{dg} \\ \\ &=  \sum_{g_r\in G_{(r-1)}U_r(k)\bs G(k)} \sum_{g_{r-1}\in G_{(r-2)}U_{r-1}(k)\bs G_{(r-1)}(k)} \cdots \sum_{g_{0}\in L(k)\bs G_{(0)}(k)} \\ &  \sum_{ \substack{S_{-r}\in  \Hom(V_{-r}, V'_{-r})(k), \\ \cdots \cdots  \\ S_{-1}\in  \Hom(V_{-1}, V'_{-1})(k), \\ y_{(0),(0)} \in Y_{(0),(0)}(k) }} (\omega( g_0\cdots g_{r-1} g_r g)\phi)(f_r)(f_{-r}+S_{-r})\cdots(f_{-1}+S_{-1})(f_0)(y_{(0),(0)}) \\ \\
&= \sum_{g_{c}\in LU(k)\bs G(k)} \\ &\sum_{ \substack{S_{-r}\in  \Hom(V_{-r}, V'_{-r})(k), \\ \cdots \cdots  \\ S_{-1}\in  \Hom(V_{-1}, V'_{-1})(k), \\ y_{(0),(0)} \in Y_{(0),(0)}(k) }} (\omega_{(r),(r+1)}( g_c g)\phi)(f_r)(f_{-r}+S_{-r})\cdots(f_{-1}+S_{-1})(f_0)(y_{(0),(0)})
\end{align*}
recalling that $U_1\cdots U_r = U$, and writing $g_c =  g_0\cdots g_{r-1} g_r$. 

We summarise this as

\begin{prop}\label{KeyComputation}
Retain all previous notation. 

If $\gamma'$ is in the image of the moment map $\phi'$, then \begin{align*} 
(\ref{Key}) = &\int_{U^{+\prime}(k)\bs U^{+\prime}(\A)} \overline{\chi_{\gamma'}(u')} \cdot\theta(\phi)(gu') \mathop{du'} \mathop{dg} \\ \\= &\sum_{g_{c}\in LU(k)\bs G(k)} \\ &\sum_{ \substack{S_{-r}\in  \Hom(V_{-r}, V'_{-r})(k), \\ \cdots \cdots  \\ S_{-1}\in  \Hom(V_{-1}, V'_{-1})(k), \\ y_{(0),(0)} \in Y_{(0),(0)}(k) }} (\omega_{(r),(r+1)}( g_c g)\phi)(f_r)(f_{-r}+S_{-r})\cdots(f_{-1}+S_{-1})(f_0)(y_{(0),(0)})
\end{align*}

Otherwise, if $\gamma'$ is not in the image of the moment map $\phi'$, then $(\ref{Key}) = 0$. 

\end{prop}

\begin{proof}
Suppose $(\ref{Key})$ is not identically zero, then one requires the existence of representatives $f_r$ and $f_{-r}$ as in Steps 3 and 4 for all $r$, which is equivalent to the existence of the moment map transfer as in Proposition \ref{prop:TransferNilpOrbit}; in particular $\gamma'$ must be in the image of the moment map $\phi'$.
\end{proof}

At this point we already obtain immediately:

\begin{thm}\label{KeyTheoremNotInImage}
Let $\pi$ be a cuspidal automorphic representation of $G$. If $\gamma'$ is not in the image of the moment map $\phi'$, then $P_{\gamma',\psi}(\theta(\cdot,\cdot), \cdot, \tau')$ is identically zero on $\Theta(\pi)$. 
\end{thm}

\noindent This is the analogous global result to \cite[Theorem 3.7(a)]{Zh}.

\vskip 5pt

In what follows, to simplify the notation, we notate the summation \[ \sum_{ \substack{S_{-r}\in  \Hom(V_{-r}, V'_{-r})(k), \\ \cdots \cdots  \\ S_{-1}\in  \Hom(V_{-1}, V'_{-1})(k), \\ y_{(0),(0)} \in Y_{(0),(0)}(k) }} \] by \[ \sum_{ \substack{S_{j}\in  H_j(k),  -r\le j\le -1  \\ y_{(0),(0)} \in Y_{(0),(0)}(k) }}. \]

\section{Relating two periods}\label{sec:RelatePeriods}

Recall from the start of Section \ref{sec:KeyComputation} that the period under computation can be written as \begin{align*}
& P_{\gamma',\psi}(\theta(\phi,f), \phi_{\gamma'}, \tau') \\
= &\int_{[G]} \overline{f(g)} \int_{[L']}  \int_{[U'/U^{+\prime}]}  \int_{[U^{+\prime}]} \theta(\phi)(g u^{+\prime}\overline{u'}l') \cdot \overline{\chi_{\gamma'}(u^{+\prime})}\mathop{du^{+\prime} } \overline{\theta_{\gamma'}(\phi_{\gamma'})(\overline{u'}l') \cdot \tau'(l') } \mathop{d\overline{u'}} \mathop{dl'} \mathop{dg} 
\end{align*}

The conclusion of Section \ref{sec:KeyComputation} is that one may replace the inner integral with the expression in Proposition \ref{KeyComputation}:

\begin{align*}
& P_{\gamma',\psi}(\theta(\phi,f), \phi_{\gamma'}, \tau') \\ \\
= &\int_{[G]} \overline{f(g)} \int_{[L']}  \int_{[U'/U^{+\prime}]} \sum_{g_{c}\in LU(k)\bs G(k)}  \\ & \sum_{ \substack{S_{j}\in  H_j(k), -r\le j\le -1  \\ y_{(0),(0)} \in Y_{(0),(0)}(k) }} (\omega( g_c g \overline{u'}l')\phi)(f_r)\cdots(y_{(0),(0)}) \cdot \overline{\theta_{\gamma'}(\phi_{\gamma'})(\overline{u'}l') \cdot\tau'(l') } \mathop{d\overline{u'}} \mathop{dl'} \mathop{dg} \\ \\
= & \int_{[G]}\sum_{g_{c}\in LU(k)\bs G(k)} \overline{f(g)} \int_{[L']}  \int_{[U'/U^{+\prime}]}  \\ & \sum_{ \substack{S_{j}\in  H_j(k), -r\le j\le -1  \\ y_{(0),(0)} \in Y_{(0),(0)}(k) }} (\omega( g_c g \overline{u'}l')\phi)(f_r)\cdots(y_{(0),(0)})\cdot \overline{\theta_{\gamma'}(\phi_{\gamma'})(\overline{u'}l') \cdot\tau'(l') } \mathop{d\overline{u'}} \mathop{dl'} \mathop{dg} \\ \\
= &\int_{LU(k)\bs G(\A)} \overline{f(g)} \int_{[L']}  \int_{[U'/U^{+\prime}]}  \\ & \sum_{ \substack{S_{j}\in  H_j(k), -r\le j\le -1  \\ y_{(0),(0)} \in Y_{(0),(0)}(k) }} (\omega( g \overline{u'}l')\phi)(f_r)\cdots(y_{(0),(0)}) \cdot\overline{\theta_{\gamma'}(\phi_{\gamma'})(\overline{u'}l') \cdot\tau'(l') } \mathop{d\overline{u'}} \mathop{dl'} \mathop{dg} \\ \\
= &\int_{LU(\A)\bs G(\A)} \int_{LU(k)\bs LU(\A)}  \overline{f(ul g)} \int_{[L']}  \int_{[U'/U^{+\prime}]}  \\ & \sum_{ \substack{S_{j}\in  H_j(k), -r\le j\le -1  \\ y_{(0),(0)} \in Y_{(0),(0)}(k) }} (\omega( ul g \overline{u'}l')\phi)(f_r)\cdots(y_{(0),(0)}) \cdot\overline{\theta_{\gamma'}(\phi_{\gamma'})(\overline{u'}l') \cdot\tau'(l') } \mathop{d\overline{u'}} \mathop{dl'}  \mathop{dl du} \mathop{dg} 
\end{align*}

where here we have performed a standard rewriting of the integral. 

\vskip 5pt

It is clear that what remains is to evaluate the respective integrals over $[U'/U^{+\prime}]$ and $[L']$. To do so, we first need to understand better the inner theta series 
\begin{equation}\label{InnerThetaSeries} \sum_{ \substack{S_{j}\in  H_j(k), -r\le j\le -1  \\ y_{(0),(0)} \in Y_{(0),(0)}(k) }}(\omega( ul g \overline{u'}l')\phi)(f_r)(f_{-r}+S_{-r})\cdots(f_{-1}+S_{-1})(f_0)(y_{(0),(0)}) \end{equation}
coming from Section \ref{sec:KeyComputation}.

\subsection{} Note that one may identify $\g_{-1}$ with the subspace of elements \[ \sum_j x_j\in \bigoplus_{j} \Hom(V_j, V_{j-1}) \] satisfying $x_j^\ast = -x_{-j+1}$ for all $j\ge 1$. 
One has the similar identification for $\g'_{-1}$. 

Recall also that one has the symplectic subspace $\Hom(V_{new}, V^{1\prime})$ of $\Hom(V,V')$, as at the end of Section \ref{NilpOrbitTransfer}.

Then one has the following key lemma (which is shown by a straightforward direct computation):

\begin{lem} \cite[Lemma 3.13]{Zh}
    There is an isomorphism of symplectic vector spaces \begin{equation}\label{KeyIsomorphismSymplecticVectorSpaces}
        -\g_{-1}\oplus \g'_{-1} \oplus \Hom(V_{new}, V^{1\prime}) \cong \bigoplus_{j=-r}^r \Hom(V_{j},V'_j) =: \mathbb{S}_{\gamma,\gamma'}. 
    \end{equation}  
    Here $-\g_{-1}$ is $\g_{-1}$ with its symplectic form negated, and we view $\Hom(V_{new}, V^{1\prime})$ and $\mathbb{S}_{\gamma,\gamma'}$ respectively as subspaces of $\Hom(V,V')$ with the inherited symplectic form. 
    
    The isomorphism on $-\g_{-1}$ is given by $x\mapsto fx$, where $f\in \Hom(V,V')$ is as in Proposition \ref{prop:TransferNilpOrbit}, and similarly the isomorphism on $\g'_{-1}$ is given by $x'\mapsto x'f$. 

\end{lem}

Note that $L$ and $L'$ act on $V_0$ and $V'_0$ respectively and hence naturally on $\mathbb{S}_{\gamma,\gamma'}$; it may be checked easily that this is equivariant for the isomorphism (\ref{KeyIsomorphismSymplecticVectorSpaces}). 

Also, there are maps from $U$ and $U'$ into the Heisenberg group associated to $\mathbb{S}_{\gamma,\gamma'}$, via the symplectic embeddings from $-\g_{-1}$ and $\g'_{-1}$ into $\mathbb{S}_{\gamma,\gamma'}$ respectively. 

Let $\Omega_{\gamma,\gamma'}$ be the Heisenberg-Weil representation associated to the symplectic space $\mathbb{S}_{\gamma,\gamma'}$. We hence see that there is an action of $LU\times L'U'$ on $\Omega_{\gamma,\gamma'}$, which is that of the Heisenberg-Weil representation pulled back to $LU\times L'U'$. 

One similarly has the key

\begin{lem}\label{lem:KeyWeilSplit}
There is a $LU\times L'U'$-equivariant isomorphism \begin{equation}\label{KeyWeilSplit} \Omega_{\gamma,\gamma'} \cong \omega_{-\gamma} \otimes \omega_{\gamma'} \otimes \omega_{\gamma,\gamma'}, \end{equation} where here: 
\begin{itemize}
    \item $-\gamma$ corresponds to the $\mfsl_2$-triple $\{-e,h,-f\}$, if $\gamma$ corresponds to the $\mfsl_2$-triple $\{e,h,f\}$; one has $\omega_{-\gamma}\cong \omega_{\gamma,\overline{\psi}}$;
    \item $\omega_{\gamma,\gamma'}$ is the Weil representation associated to $\Hom(V_{new}, V^{1\prime})$, as at the end of Section \ref{NilpOrbitTransfer}.
\end{itemize}
\end{lem}

Observe that $\Omega_{\gamma,\gamma'}$ may be realised on the space of Schwarz functions on the Lagrangian \[ \Y_{\gamma,\gamma'} := \bigoplus_{j=-r}^{-1} \Hom(V_{j},V'_j) \oplus Y_{(0),(0)} \] of $\mathbb{S}_{\gamma,\gamma'}$.

\vskip 10pt

Following \cite[Section 6.3]{GZ}, we now make the following definition:

\begin{definition}
Let $\Phi(g)$ denote the Schwarz function on $\Y_{\gamma,\gamma'}$, and hence an element of $\Omega_{\gamma,\gamma'}$, defined by \begin{equation}\label{Phi} \Phi(g)(S_{-r},\cdots,S_{-1},y_{(0),(0)}) := (\omega(g)\phi)(f_r, f_{-r}+S_{-r},\cdots,f_{-1}+S_{-1},f_0, y_{(0),(0)}). \end{equation}

Here, note of course that the $f_j$ are all elements over the field $k$, so that $\Phi(g)$ is still a well-defined Schwarz function on the adelic space $\Y_{\gamma,\gamma'}(\A)$.
\end{definition}

\vskip 5pt

\begin{lem}\label{lem:LULUEquivariance1}
For fixed $g$, the map \begin{align*}
    \omega&\longrightarrow \Omega_{\gamma,\gamma'} \\
    \omega(g)\phi &\longmapsto \Phi(g)
\end{align*}
is $LU\times L'U'$-equivariant.
\end{lem}
\begin{proof}
This follows from \cite[Proposition 6.5]{GZ}, or \cite[Proposition 3.8]{Zh}. 
\end{proof}

\vskip 5pt

\subsection{Theta series} By the preceding Lemma \ref{lem:LULUEquivariance1}, one may therefore write 
\begin{align}  
    (\ref{InnerThetaSeries}) =  \sum_{y_{\gamma,\gamma'} \in \Y_{\gamma,\gamma'}(k) } (\Omega_{\gamma,\gamma'}(ul\overline{u'}l')\Phi(g))(y_{\gamma,\gamma'}) \nonumber
\end{align}

for $\Phi(g)\in\Omega_{\gamma,\gamma'}$. 

\vskip 5pt

Now, since $\Phi(g)\in\Omega_{\gamma,\gamma'}$, then using Lemma \ref{lem:KeyWeilSplit}, let us by an abuse of notation\footnote{in other words, we omit the relevant canonical intertwiners of the Heisenberg-Weil representation $\Omega_{\gamma,\gamma'}$, so as to keep the notation manageable} view $\Phi(g)$ as an element of $\omega_{-\gamma} \otimes \omega_{\gamma'} \otimes \omega_{\gamma,\gamma'}$.

That is, we view $\Phi(g)$ as a Schwarz function on $Y_{-1}\oplus Y'_{-1}\oplus Y_{\gamma,\gamma'}$, for Lagrangian subspaces $Y_{-1}, Y'_{-1}, Y_{\gamma,\gamma'}$ of $\g_{-1},\g'_{-1},\Hom(V_{new}, V^{1\prime})$ respectively. 

We thus write \begin{align}  
    (\ref{InnerThetaSeries}) =  \sum_{\substack{y_{\gamma} \in Y_{-1}(k) \\ y_{\gamma'} \in Y'_{-1}(k) \\ y_{\gamma,\gamma'} \in Y_{\gamma,\gamma'}(k)} } (\Omega_{\gamma,\gamma'}(ul\overline{u'}l')\Phi(g))(y_{\gamma}+y_{\gamma'}+y_{\gamma,\gamma'}) \nonumber
\end{align}

\vskip 5pt

\subsection{Integral over $[U'/U^{+\prime}]$} 

In order to evaluate the inner integral over $[U'/U^{+\prime}]$, the pertinent evaluation is that of
\begin{equation}
\label{g'1Integral}
\int_{[U'/U^{+\prime}]} \sum_{\substack{y_{\gamma} \in Y_{-1}(k) \\ y_{\gamma'} \in Y'_{-1}(k) \\ y_{\gamma,\gamma'} \in Y_{\gamma,\gamma'}(k)} } (\Omega_{\gamma,\gamma'}(ul\overline{u'}l')\Phi(g))(y_{\gamma}+y_{\gamma'}+y_{\gamma,\gamma'}) \cdot \overline{\theta_{\gamma'}(\phi_{\gamma'})(\overline{u'}l')} \mathop{d\overline{u'}}
\end{equation}

As the theta series here involve the Heisenberg group action of $U'$ on the Heisenberg-Weil representation $\omega_{\gamma'}$, one should break the integral over maximally isotropic complementary subspaces of $\g'_{-1}=X'_{-1}\oplus Y'_{-1}$, supposing that $\omega_{\gamma'}$ is realised on the space of Schwarz functions on $Y'_{-1}$.

One has 
\begin{align*}
(\ref{g'1Integral}) = &\int_{[Y'_{-1}]}\int_{[X'_{-1}]} \sum_{\substack{ y_{\gamma} \in Y_{-1}(k) \\ y_{\gamma'} \in Y'_{-1}(k) \\ y_{\gamma,\gamma'} \in Y_{\gamma,\gamma'}(k) \\ y_{\gamma',2} \in Y'_{-1}(k)} } \psi(\langle x', y_{\gamma'}\rangle - \langle x', y_{\gamma',2}\rangle) \\ &\cdot (\Omega_{\gamma,\gamma'}(ull') \Phi(g) )(y_{\gamma}+y_{\gamma'}+y'+y_{\gamma,\gamma'}) \cdot \overline{ (\omega_{\gamma'}(l') \phi_{\gamma'})(y_{\gamma',2}+y') } \mathop{dx' dy'}
\end{align*}

Evaluating first the inner integral over $[X'_{-1}]$, again observe that only the terms with $y_{\gamma'}=y_{\gamma',2}$ contribute to the sum, and we obtain: 

\begin{align*}
(\ref{g'1Integral}) &= \int_{[Y'_{-1}]} \sum_{\substack{ y_{\gamma} \in Y_{-1}(k) \\ y_{\gamma'} \in Y'_{-1}(k) \\ y_{\gamma,\gamma'} \in Y_{\gamma,\gamma'}(k) }} (\Omega_{\gamma,\gamma'}(ull') \Phi(g) )(y_{\gamma}+y_{\gamma'}+y'+y_{\gamma,\gamma'}) \cdot \overline{ (\omega_{\gamma'}(l') \phi_{\gamma'})(y_{\gamma'}+y') } \mathop{dy'} \\
&= \int_{Y'_{-1}(\A)} \sum_{\substack{ y_{\gamma} \in Y_{-1}(k) \\ y_{\gamma,\gamma'} \in Y_{\gamma,\gamma'}(k) }} (\Omega_{\gamma,\gamma'}(ull') \Phi(g) )(y_{\gamma}+y'+y_{\gamma,\gamma'}) \cdot \overline{ (\omega_{\gamma'}(l') \phi_{\gamma'})(y') } \mathop{dy'}
\end{align*}

\subsection{Integral over $[L']$} What remains is therefore the evaluation of the integral 
\begin{align}
\label{L'Integral}
&\int_{[L']} \int_{Y'_{-1}(\A)} \sum_{\substack{ y_{\gamma} \in Y_{-1}(k) \\ y_{\gamma,\gamma'} \in Y_{\gamma,\gamma'}(k) }} (\Omega_{\gamma,\gamma'}(ull') \Phi(g) )(y_{\gamma}+y'+y_{\gamma,\gamma'}) \cdot \overline{ (\omega_{\gamma'}(l') \phi_{\gamma'})(y') } \mathop{dy'} \cdot \overline{\tau'(l') }  \mathop{dl'} 
\end{align}

Define
\begin{equation*}
    \Big(\Phi(g)\ast \overline{  \phi_{\gamma'} } \Big)(y_{\gamma}+y_{\gamma,\gamma'})  := \int_{Y'_{-1}(\A)}  \Phi(g) (y_{\gamma}+y'+y_{\gamma,\gamma'}) \cdot \overline{  \phi_{\gamma'}(y') } \mathop{dy'} 
\end{equation*}

which is a Schwarz function on the Lagrangian $Y_{-1}\oplus Y_{\gamma,\gamma'}$, hence an element of $\omega_{-\gamma}\otimes \omega_{\gamma,\gamma'}$. 

Then it is straightforward to see that we may write
\begin{align*}
   (\ref{L'Integral}) &= \int_{[L']} \sum_{\substack{ y_{\gamma} \in Y_{-1}(k) \\ y_{\gamma,\gamma'} \in Y_{\gamma,\gamma'}(k) }}  \big(\Omega_{\gamma,\gamma'}(ull') ( \Phi(g)\ast \overline{  \phi_{\gamma'} })  \big)(y_{\gamma}+y_{\gamma,\gamma'})   \cdot \overline{\tau'(l') }  \mathop{dl'} 
\end{align*}

Write now \[ \Phi(g)\ast \overline{  \phi_{\gamma'}} = \sum \, (\Phi(g)\ast \overline{  \phi_{\gamma'}})_\gamma \otimes (\Phi(g)\ast \overline{  \phi_{\gamma'}})_{\gamma,\gamma'} \, \, , \] a sum of pure tensors in $\omega_{-\gamma}\otimes \omega_{\gamma,\gamma'}$. 

\vskip 5pt

Then it is clear that 
\begin{align*}
   (\ref{L'Integral}) &= \sum \, \theta_{-\gamma}\big( (\Phi(g)\ast \overline{  \phi_{\gamma'}})_\gamma \big)(ul) \cdot \theta_{\gamma,\gamma'}\Big(\big( (\Phi(g)\ast \overline{  \phi_{\gamma'}})_{\gamma,\gamma'}\big),\ \tau' \Big)(l)
\end{align*}

\subsection{Relating two periods}

The conclusion of the above discussion is that 

\begin{thm}\label{thm:KeyResult}
Retain all previous notation. Then
\begin{align}
\label{KeyResult}
& P_{\gamma',\psi}(\theta(\phi,f), \phi_{\gamma'}, \tau') \\
= &\int_{LU(\A)\bs G(\A)} \int_{LU(k)\bs LU(\A)}  \overline{f(ul g)}\cdot \sum \, \theta_{-\gamma}\big( (\Phi(g)\ast \overline{  \phi_{\gamma'}})_\gamma \big)(ul) \nonumber \\ & \cdot \theta_{\gamma,\gamma'}\Big(\big( (\Phi(g)\ast \overline{  \phi_{\gamma'}})_{\gamma,\gamma'}\big),\ \tau' \Big)(l)  \mathop{dl du} \mathop{dg}  \nonumber \\
= & \int_{LU(\A)\bs G(\A)} \sum \overline{P_{\gamma,\overline{\psi}}\bigg( g\cdot f,\ (\Phi(g)\ast \overline{  \phi_{\gamma'}})_\gamma  ,\ \theta_{\gamma,\gamma'}\Big((\Phi(g)\ast \overline{  \phi_{\gamma'}})_{\gamma,\gamma'} \, , \, \tau' \Big) \bigg) } \mathop{dg} \nonumber
\end{align}
\end{thm}

We obtain immediately
\begin{thm}\label{KeyTheorem}
Suppose $\gamma'$ is in the image of the moment map $\phi'$. Let $\pi$ be a cuspidal automorphic representation of $G$ and $\tau'$ a cusp form on $L'$. If $P_{\gamma',\psi}(\theta(\cdot,\cdot), \cdot, \tau')$ is not identically zero on $\Theta(\pi)$, then $P_{\gamma,\overline{\psi}}( \cdot, \cdot  ,\theta_{\gamma,\gamma'}(\cdot, \tau' ))$ is not identically zero on $\pi$.  
\end{thm}

\vskip 5pt

\subsection{Even case}

Since the relation in (\ref{KeyResult}) is slightly unwieldy in full generality, it is instructive to explicate the relation in the even case. 

Suppose that $\gamma,\gamma'$ are both even, so that $\g_{-1},\g'_{-1}$ are both trivial. Suppose further also that $L'$ is trivial. These assumptions are only made so that one does not have to work with theta series in this case, and includes all the (even) `hook-type' cases as considered in \cite[Section 7]{GW}. 

Then in this case (\ref{KeyResult}) takes the simpler form 

\begin{align}
\label{KeyResultSimple}
& P_{\gamma',\psi}(\theta(\phi,f)) = \int_{[L']} \int_{[U']} \theta(\phi,f)(u'l') \cdot \overline{ \chi_{\gamma'}(u')  } \mathop{du'} \mathop{dl'} \\
= & \int_{LU(\A)\bs G(\A)} (\omega_{(r),(r+1)}(g)\phi)(f_r)(f_{-r})\cdots (f_1)(f_{-1})(f_0) \cdot \overline{P_{\gamma,\overline{\psi}}( g\cdot f) } \mathop{dg} \nonumber\\
\label{KeyResultSimpleIntegral} = & \int_{LU(\A)\bs G(\A)} \Phi(g) \cdot \overline{P_{\gamma,\overline{\psi}}( g\cdot f) } \mathop{dg} 
\end{align}

\vskip 5pt

\subsection{Converse of Theorem \ref{KeyTheorem}} One would now like to show the converse of Theorem \ref{KeyTheorem}. 

It is instructive to focus first on the even case, that is, we want to show that if $P_{\gamma,\overline{\psi}}( \cdot)$ is not identically zero on $\pi$, then $P_{\gamma',\psi}(\theta(\cdot,\cdot))$ is not identically zero on $\Theta(\pi)$. 

First, one may certainly suppose $P_{\gamma,\overline{\psi}}( f)\ne 0$, that is, \[ g\mapsto P_{\gamma,\overline{\psi}}( g\cdot f) \] is non-vanishing in a neighbourhood of the identity, and the aim is to show that one may choose $\Phi$ such that the integral (\ref{KeyResultSimpleIntegral}) is non-vanishing. 

Now since $\phi$ is an arbitrary Schwarz function, then analogously to \cite[Proposition 7.1]{CG}, the results of \cite[Proposition 6.5]{GZ} show that \[ \Phi : g\mapsto (\omega_{(r),(r+1)}(g)\phi)(f_r)(f_{-r})\cdots (f_1)(f_{-1})(f_0)\] is an arbitrary Schwarz function on $LU\bs G$, at least in the local setting. 

However, in the global setting, one does not immediately have the analogous conclusion to \cite[Proposition 7.1]{CG} (since $LU\bs G$ is not closed in $\Hom(V,V')$). What is pertinent here is therefore the computation at the unramified places. 

Let $S$ be the finite set of places of $k$, including the Archimedean places, such that $f$ and $\psi$ are unramified outside of $S$. For each place $v\notin S$, we would like to choose $\phi_v$ the characteristic function of $Y_{(r),(r+1)}(\mf{o}_{k_v})$.

Recall (section \ref{NilpOrbits}) that one has the parabolic subgroup $P=MU$ of $G$. Furthermore one has, in the notation of section \ref{WeightSpace}, \[ M = M_0\times M_1\times\cdots\times M_{r} \cong G(V_0) \times \GL(V_{-1})\times\cdots\times \GL(V_{-r}). \] Accordingly, for each $m\in M$, we write $m=m_0m_1\cdots m_r$. Recall also (cf. section \ref{NilpOrbitTransfer}) that $L=G(\ker f_0)\subseteq G(V_0)=M_0$. 

For each place $v\notin S$, by the Iwasawa decomposition, $\Phi$ is specified by its value on $L(k_v) \bs M(k_v)$. Now combining (suitable analogues of) the formulas (\ref{Graction}), (\ref{Mr+1action}) and (\ref{Gr-1action}) describing the action of the Weil representation, one readily sees that \begin{align}
    &(\omega_{(r),(r+1)}(m_0m_1\cdots m_r)\phi_v)(f_r)(f_{-r})\cdots (f_1)(f_{-1})(f_0) \\
    = &\nu(m) \cdot \phi_v(f_r m_r)(f_{-r} m_r)\cdots (f_1 m_1)(f_{-1} m_1)(f_0 m_0). \nonumber
\end{align}

Observe that, for each $1\le j\le r$, under the identification $M_j\cong \GL(V_{-j})$, the element $m_j$ is effectively acting on $f_j\in \Hom(V_j, V'_{j+1})$ by its matrix inverse (transpose), and acting on $f_{-j}\in \Hom(V_{-j}, V'_{-(j-1)})$ by itself. 

Furthermore we may WLOG assume we have chosen $f_j$ and $f_{-j}$, and suitable bases, such that they are represented by identity matrices (or matrices with only 1s in some diagonal entries, and 0s in all other entries), cf. \cite[Lemma 3.4]{Zh}.  

Finally, note also that, as in section \ref{BaseCase}, $L$ is the stabiliser of $f_0$ under the $M_0$-action.

Since $\phi_v$ is the characteristic function of $Y_{(r),(r+1)}(\mf{o}_{k_v})$, one therefore sees immediately that $\Phi$, and hence the entire integrand, is supported on \[ L\bs M_0(\mf{o}_{k_v}) \times \GL(V_{-1})(\mf{o}_{k_v})\times\cdots\times \GL(V_{-r})(\mf{o}_{k_v}) \] for $\mf{o}_{k_v}$ the ring of integers of $k_v$.

Therefore, with this choice of $\phi$ and hence $\Phi$, analogously to \cite[p.2719]{GS}, the integral reduces to an integration over (groups with coefficients in) $\A_S$, and one may then use as above that $\Phi$ is an arbitrary Schwarz function on $LU\bs G$ to complete the argument. 

\vskip 5pt

The argument in the general case is similar, since the main point is really the support of the Schwarz function, and hence the integrand, at the unramified places $v\not\in S$. For the remaining places $v\in S$, one may multiply $\Phi(g)$ by an arbitrary Schwarz function on $LU\bs G$. In summary, one has

\begin{thm}\label{KeyTheoremConverse}
Suppose $\gamma'$ is in the image of the moment map $\phi'$. Let $\pi$ be a cuspidal automorphic representation of $G$ and $\tau'$ a cusp form on $L'$. If $P_{\gamma,\overline{\psi}}( \cdot, \cdot  ,\theta_{\gamma,\gamma'}(\cdot, \tau' ))$ is not identically zero on $\pi$, then $P_{\gamma',\psi}(\theta(\cdot,\cdot), \cdot, \tau')$ is not identically zero on $\Theta(\pi)$. 
\end{thm}

\vskip 5pt

\section{Examples}\label{sec:Examples}

In this section, we give several examples to illustrate how the results above may be commonly used to relate certain well-known periods. 

\subsection{Bessel / Fourier-Jacobi to Whittaker}

Suppose $(G, G') = (\O_{2n}, \Sp_{2m})$ with $n\ge m$, and $\gamma'$ corresponds to a regular nilpotent orbit in $G'$, with corresponding partition $\lambda' = [2m]$. Then $\gamma$ has corresponding partition $\lambda = [2m-1, 1^{2n-2m+1}]$, and we have $L=\O(V_{new})\cong \O_{2n-2m+1}$ for the subspace $V_{new}\subseteq V$ of dimension $2n-2m+1$. The $\gamma$-period $P_{\gamma}$ on $\pi$ corresponds to the Bessel period on $\pi$ \cite{GGP}. 

Therefore, in this case, we are relating the Bessel period on $\pi$ to the Whittaker period on $\Theta(\pi)$. This is essentially the case considered in \cite[Section 7]{GW}, with several more very similar cases considered there as well. For example, one may similarly relate the Fourier-Jacobi period on $\Sp_{2n}$ to the Whittaker period on $\O_{2m}$. 

\subsection{Bessel / Fourier-Jacobi to Bessel / Fourier-Jacobi}

From the preceding example (cf. also \cite[Section 7]{GW}), we see that partitions of the general form $\lambda = [k, 1^{a}]$ (for $k>1$), or the so-called `hook-type' partitions, correspond either to Bessel or Fourier-Jacobi periods. 

It is easy to see that if $\gamma'$ is chosen to have `hook-type' partition $\lambda' = [k, 1^{a}]$ (and is in the image of the moment map $\phi'$), then $\gamma$ will also have `hook-type' partition $\lambda= [k-1,1^b]$ (for some $b$). If $k\ge 3$, $a\ge 2$ and $b\ge 2$, then one sees that we are able to relate a Bessel / Fourier-Jacobi period to another Bessel / Fourier-Jacobi period in a specific way. 

If $k=2$, then $G'$ must be symplectic and $G$ must be orthogonal, and the $\gamma$-period on $G$ is the reductive Bessel period on $G$ (i.e. the subspace $V_{new}\subseteq V$ is of codimension 1). 

\subsection{Whittaker to Bessel / Fourier-Jacobi}

If, in the preceding section, $b=0$ ($G$ symplectic) or 1 ($G$ orthogonal), then one sees that we are relating a Whittaker period on $\pi$ to a Bessel or Fourier-Jacobi period on $\Theta(\pi)$.

\subsection{Shalika to reductive periods}

Finally, suppose $\gamma'$ has corresponding partition $\lambda' = [2^a]$ for some $a$, then the $\gamma'$-period on $G'$ is the so-called Shalika period. Furthermore $\gamma$ must correspond to the trivial partition $\lambda = [1^{a+b}]$ for some $b$, with $L\cong G(V_{new})$ and $V_{new}$ of dimension $b$. 

This case has been used in various settings to relate the Shalika period on $\Theta(\pi)$ to the reductive $G(V_{new})$-period on $\pi$. When $a=b$ for instance, we are relating the Shalika period on $\Theta(\pi)$ to the so-called linear or Friedberg-Jacquet period on $\pi$.

\end{document}